\documentclass[preprint,11pt]{amsart} 
\usepackage{amssymb, amscd, verbatim, amsmath, MnSymbol}
 \newtheorem{theorem}{Theorem}
\newtheorem{lemma}[theorem]{Lemma}
\newtheorem{corollary}[theorem]{Corollary}
\newtheorem{proposition}[theorem]{Proposition}

\def\tr{\Delta}
\def\bea{\begin{eqnarray*}}
\def\eea{\end{eqnarray*}}
\def\be{\begin{eqnarray}}
\def\ee{\end{eqnarray}}

\numberwithin{equation}{section}

\begin{document}

\title[Gap theorems on critical point equation]{Gap theorems on critical point equation of the total scalar curvature with divergence-free Bach tensor}
\author{Gabjin Yun}
\address{Department of Mathematics\\  Myong Ji University\\
116 Myongji-ro Cheoin-gu\\ Yongin, Gyeonggi 17058, Republic of Korea. }
\email{gabjin@mju.ac.kr}

\author{Seungsu Hwang}
\address{Department of Mathematics\\ Chung-Ang University\\
84 HeukSeok-ro DongJak-gu \\ Seoul 06974, Republic of Korea.
}
\email{seungsu@cau.ac.kr} 



 \begin{abstract}
 On a compact $n$-dimensional manifold, it is well known that a critical metric of the total scalar curvature, restricted to the space of metrics with unit volume is Einstein. 
It  has been conjectured that a critical metric of the total scalar curvature, restricted to the
space of metrics with constant scalar curvature of unit volume, will be Einstein. This conjecture, proposed in 1987 by Besse, has not been resolved except when $M$ has harmonic curvature or the metric is Bach flat.
In this paper, we prove some gap properties under divergence-free Bach tensor condition for $n\geq 5$, and a similar condition for $n=4$. \end{abstract}

\maketitle
\section{Introduction}

Let $M$ be an $n$-dimensional compact manifold, and let ${\mathcal M}_1$ be the set of all smooth 
Riemannian structures of unit volume on $M$. The total scalar curvature ${\mathcal S}$ on ${\mathcal M}_1$ is
 given by $$ {\mathcal S}(g)=\int_M s_g \, dv_g,$$ 
where $s_g$ is the scalar curvature of $g\in {\mathcal M}_1$. Hilbert showed that critical points of
 ${\mathcal S}$ on ${\mathcal M}_1$ are Einstein. In \cite{Ko},  Koiso 
introduced the space ${\mathcal C}$ of constant scalar curvature metrics of unit volume.
 The Euler-Lagrange equation of ${\mathcal S}$ restricted to ${\mathcal C}$ may be written
 in the form of the following critical point equation 
\be
z_g=s_g'^*(f).\label{cpe0}
\ee
Here, $z_g$ is the traceless Ricci tensor corresponding to $g$, and the operator $s_g'^*$ is the $L^2$
 adjoint of the linearization $s_g'$ of the scalar curvature, given by 
$$ s_g'^*(f)=D_gdf  - (\tr_g f)g-fr_g,$$
where $D_gd$ and $\tr_g$ denote the Hessian and the (negative) Laplacian, respectively, and $r_g$ is
 the Ricci curvature of $g$. 
 If $f=0$ in (\ref{cpe0}), then $g$ is clearly Einstein. 
By taking the trace of (\ref{cpe0}), we obtain 
$$ \tr_g f =-\frac {s_g}{n-1}f.$$
Thus, if $ s_g/{(n-1)}$ is not in the spectrum of $\tr_g$, then the critical metric $g$ is again Einstein. 
For example, if $s_g\leq 0$, then $g$ is Einstein. 
Note that if a non-trivial solution $(g,f)$ of (\ref{cpe0}) is Einstein,
then (\ref{cpe0}) is reduced to  the Obata equation, and so $(M, g)$  should be isometric to a standard $n$-sphere (\cite{Ob}). 

We remark that the existence of a non-trivial solution is a strong condition. 
The only known case satisfying this is that of the standard sphere.
It was conjectured in \cite{Be} that this is the only possible case.
 
\vskip .5pc
{\sc Besse Conjecture.} {\it Let $(g,f)$ be a solution of {\rm (\ref{cpe0})} on an $n$-dimensional 
compact manifold $M$. Then, $(M,g)$ is Einstein.}
\vskip .5pc

There are some partial answers to this conjecture.
For example, it was proved that the Besse conjecture holds if $M$ has harmonic curvature
 (see Theorem 1.2 of \cite{ych2} and also \cite{erra}). A Riemannain manifold $(M,g)$ is said to
 have harmonic curvature if $\delta R=0$, where $R$ is the full Riemann tensor,
 and  $\delta$ is the negative divergence operator.
 In particular, a locally conformally flat non-trivial solution $(g,f)$ of (\ref{cpe0}) with $s_g>0$ is
 clearly isometric to a standard sphere. 
 Note that when $s_g$ is constant, $\delta R=0$ if and only if $\delta {\mathcal W}=0$ (cf. (\ref{eqn02}) below), 
where ${\mathcal W}$ is the Weyl tensor. Qing and Yuan showed in \cite{qy} that the Besse conjecture  holds
 if $g$ is Bach-flat, i.e., $B=0$, where $B$ is the $n$-dimensional Bach tensor (see Section 2 for its definition). In fact, they proved that Bach-flatness implies harmonic curvature. Thus, it is natural to consider the divergence-free Bach tensor condition in the critical point equation (\ref{cpe0}) as a way to generalize Bach-flat condition.
 It turns out that $\delta B$ vanishes automatically when $n=4$, and $\delta B=0$ if and only if
 $\langle i_XC, z_g\rangle =0$ for any vector $X$ when $n\geq 5$ (see Proposition~\ref{lem02} below).
 Here, $C$ is the Cotton tensor defined by (\ref{eqn03}) below.  

In this paper, we will prove some gap properties under the assumption $\delta B=0$. 
For a non-trivial solution $(g, f)$ of (\ref{cpe0}), we define $\mu$ by
$$ 
\mu=\max \left\{ |\min_M (1+f)|, \max_M (1+f)\right\}.
$$
If $f$ satisfies $f \ge -1$, we can easily show that $(M, g)$ is Einstein. In fact, one can show from (\ref{cpe0}) that ${\rm div}(z_g(\nabla f, )) = (1+f)|z_g|^2$, and so rigidity follows from the divergence theorem. Note that $\mu \geq  1$, because we have 
$$
0=\int_M \tr f =-\frac s{n-1}\int_M f, 
$$ 
which implies that  there exists a point $p\in M$ satisfying $f(p)=0$. In fact, $\mu >1$ unless $f$ is trivial.

Our first main result for gap property on the critical point equation is the following.

\begin{theorem}  \label{main1} 
Let $(g,f)$ be a non-trivial solution of {\rm (\ref{cpe0})} on an $n$-dimensional 
compact manifold $M$. Assume that $\langle i_XC, z_g\rangle=0$ for any vector $X$ and $n\geq 4$.
 If $|z_g|^2\leq \frac {s_g^2}{4 n(n-1)\mu^2}$, then $(M,g)$ is isometric to a standard $n$-sphere.
\end{theorem}

As mentioned above, when $n \ge 5$, the condition that the Bach tensor is  divergence-free implies
the first hypothesis in Theorem~\ref{main1}.
Thus, for $n\geq 5$ we have the following result.

\begin{corollary}  Let $n\geq 5$ and $(g,f)$ be a non-trivial solution of {\rm (\ref{cpe0})} on an $n$-dimensional compact  manifold $M$ having divergence-free Bach tensor.  If $|z_g|^2\leq \frac {s_g^2}{4 n(n-1)\mu^2}$, then $(M,g)$ is isometric to a standard $n$-sphere.
\end{corollary}

In  \cite{ych2} and  \cite{erra}, we proved the Besse conjecture is true when $(M, g)$ has harmonic curvature.
In this case, the traceless Ricci tensor $z_g$ can be decomposed into $\nabla f$-direction and its orthogonal complement. In other words, for a vector $X$ orthogonal to  $\nabla f$, we have
$z_g(\nabla f, X) = 0$, and so $z_g$ can be controlled by $z_g(N, \cdot) = i_Nz_g$ with $N = \frac{\nabla f}{|\nabla f|}$ on each hypersurface 
given by a level set of $f$. Related to $i_Nz_g$, we have the following gap property.

\begin{theorem}\label{main2}
Let $(g,f)$ be a non-trivial solution of {\rm (\ref{cpe0})} on an $n$-dimensional compact manifold $M$. 
Assume that $\langle i_XC, z_g\rangle=0$ for any vector $X$ and $n\geq 4$. 
If $$|z_g|^2 \leq \min \left\{ 2|i_Nz_g|^2, \frac {s_g^2}{4n(n-1)}\right\},$$
then $(M,g)$ is isometric to a standard sphere.
\end{theorem}
It is comparable with Theorem 2 of \cite{Bal}, which states that a non-trivial solution $(g,f)$ of (\ref{cpe0}) 
has zero radial Weyl curvature with $$ |z_g|^2 \leq \frac {s_g^2}{n(n-1)},$$ then $(M,g)$ is isometric to a standard sphere.
We say that $g$ has zero radial Weyl curvature if $\tilde{i}_{\nabla f}{\mathcal W}=0$, 
where $\tilde{i}_X$ is defined in (\ref{lateri}).

This paper is organized as follows.
In Section 2, we give some properties of Bach tensor and Cotton tensor with their divergences. In particular, we include the fact that the divergence of Bach tensor is given by the inner product of the Cotton tensor with traceless Ricci tensor (Proposition~\ref{lem02}). In Section 3, we introduce a covariant $3$-tensor and derive some properties of it to handle the critical point equation. In section 4 and 5, we prove our main results, Theorem~\ref{main1} and \ref{main2}.

\section{Divergences of a Bach tensor}

Let $(M, g)$ be an $n$-dimensional Riemannian manifold. For convenience, we denote $s_g, r_g$
 and $z_g$ by $s$, $r$ and $z$, respectively, if there is no ambiguity. 
 Throughout the paper, we will assume that the dimension $n\geq 4$. 

Let $D$ be the Levi-Civita connection on $(M, g)$ and let us denote by $C^{\infty}(S^2M)$ the space of sections of symmetric $2$-tensors on $(M,g)$. 
 Then, the differential operator $d^D$ from $C^{\infty}(S^2M)$ to 
  $C^\infty\left(\Lambda^2 M \otimes T^*M\right)$ is defined by
$$ 
d^D \eta(X,Y,Z)= (D_X \eta)(Y,Z)-(D_Y \eta)(X,Z)
$$
for $\eta \in C^{\infty}(S^2M)$ and vectors $X, Y$, and $Z$. In particular, the following result is well known (\cite{Be}): under the identification of $C^\infty(T^*M \otimes \Lambda^2M)$ with
$C^\infty(\Lambda^2M \otimes T^*M)$,
\be 
\delta R=-d^Dr. \label{eqn02}
\ee

For a function $\varphi \in C^\infty(M)$ and $\eta \in C^{\infty}(S^2M)$,
$d\varphi \wedge \eta $ is defined by
$$
(d\varphi \wedge \eta) (X,Y,Z)= d\varphi(X) \eta(Y,Z)-d\varphi(Y) \eta(X,Z).
$$
Here, $d\varphi$ denotes the usual total differential of $\varphi$.

The Cotton tensor $C\in \Gamma(\Lambda^2M\otimes T^*M)$ is defined by
\be C=d^Dr -\frac 1{2(n-1)}\, {ds}\wedge g\label{eqn03}\ee
and
 the $n$-dimensional Bach tensor $B$ is defined by
$$ B=\frac 1{n-3}\, \delta^D\delta {\mathcal W}+ \frac 1{n-2}\, \mathring{\mathcal W}z.
$$
Here, 
$\delta^D$ is the $L^2$ adjoint operator of $d^D$, and
$$\mathring{\mathcal W}z(X,Y)=\sum_{i=1}^n z({\mathcal W}(X, E_i)Y, E_i)$$
for an orthonormal frame $\{E_i\}_{i=1}^n$. $\mathring{R}r$ is defined similarly.
From now on, we will omit the summation notation, as we employ the Einstein convention.

Because
$$ \delta {\mathcal W}= -\frac {n-3}{n-2}\, d^D\left(r-\frac {s}{2(n-1)}\, g\right),$$
we have
\be C= -\frac {n-2}{n-3}\, \delta {\mathcal W} \quad \mbox{and} \quad
\delta C= -\frac {n-2}{n-3} \delta^D\delta {\mathcal W}. \label{eqn04}
\ee
As a consequence, we have
\be (n-2) B=  -\delta C +\mathring{\mathcal W}z . \label{eqn05}
\ee

\begin{proposition}{\rm (Corollary 1.22 of \cite{Be})} \label{lem01} 
For any tensor $h$, we have
$$
D^2_{X,Y} h -D^2_{Y,X}h= -R(X,Y)h
$$
and
$$ D^3_{X,Y,Z}h -D^3_{Y,X,Z}h= -R(X,Y)D_Z h +D_{R(X,Y)Z}h.$$
\end{proposition}

Recall that the Schouten tensor $A$ is defined by 
$$
A= r -\frac {s}{2(n-1)}g$$ 
so that $C=d^DA$. 
The following is Lemma 5.1 of \cite{CC}.   We include the proof for the sake of completeness.

\begin{proposition} \label{lem02} 
For any vector field $X$ we have
$$
(n-2) \delta B (X) = -\frac {n-4}{n-2} \, \langle i_XC, z\rangle = -\frac {n-4}{n-2}\left( \frac 12\, X(|z|^2)+ \delta (z\circ z)(X)\right).
$$
Here,  $$i_XC(Y,Z)=C(X,Y,Z),$$
$$ \langle i_XC, z\rangle = \sum_{i,j=1}^n i_XC(E_i, E_j)z(E_i,E_j),$$
and a $2$-tensor $z\circ z$ is defined by 
$$ z\circ z(X, Y)=\sum_{i=1}^n z(X, E_i)z(E_i, Y)$$
for any orthonormal frame $\{E_i\}_{i=1}^n$ and vector fields $X$, $Y$, and $Z$. 
\end{proposition}
\begin{proof}
Let  $\{E_i \}_{i=1}^n$ be a geodesic frame. Denoting $z_{ij} = z(E_i, E_j)$ and $r_{ij}=r(E_i, E_j)$, 
it follows from (\ref{eqn04}) and (\ref{eqn05}) that
\bea 
(n-2) \delta B (X)
&= &  -\delta \delta C(X)+ \delta \mathring{\mathcal W}z(X)\\
&=& -\delta \delta C(X)-\frac {n-3}{n-2}C(X, E_i, E_j)z_{ij}
+ \frac 12 {\mathcal W}(E_i, E_j, E_k, X)C(E_i, E_j, E_k).
\eea
Note that
\bea
\delta\delta C(X)&=& \delta \delta d^DA(X) \\
&=& D_{E_i}D_{E_k}(D_{E_k} A(E_i, X)- D_{E_i}  A(E_k, X)) \\
&=&   (D_{E_i}D_{E_k}-D_{E_k}D_{E_i})D_{E_k}A(E_i, X).
\eea
Thus, by Proposition~\ref{lem01} we have
\bea
\delta\delta C(X)&=&R(E_k, E_i)D_{E_k}A(E_i, X)-D_{R(E_k, E_i)E_k} A(E_i, X) \\
&=& - D_{E_k}A(R(E_k, E_i)E_i, X) -D_{E_k}A(E_i, R(E_k, E_i) X) - r_{is}D_{E_s}A(E_i, X)\nonumber\\
&=& r_{kj}D_{E_k}A(E_j, X) +   \langle R(E_k, E_i) E_s, X)D_{E_k}A(E_i, E_s)- r_{ik}D_{E_k}A(E_i, X)\\
&=&   \frac 12 \langle R(E_k, E_i) E_s, X)C(E_k, E_i, E_s). \nonumber
\eea
Hence,  
$$
(n-2)\delta B(X)= -\frac {n-3}{n-2}C(X, E_i, E_j)z_{ij}  +\frac 12 ({\mathcal W}-R)(E_i, E_j, E_k, X)C(E_i, E_j, E_k).$$
From the decomposition of Riemann tensor, it follows 
$${\mathcal W}_{ijkl}= R_{ijkl} -\frac 1{n-2}( g_{ik}A_{jl}+g_{jl}A_{ik}- g_{jk}A_{il}-g_{il}A_{jk} ), 
$$
and so 
$$ ({\mathcal W}-R)(E_i, E_j, E_k, E_l)C(E_i, E_j, E_k)=-\frac 2{n-2} (A_{jl}C_{iji}+A_{ik}C_{ilk})= -\frac 2{n-2}r_{ik}C_{ilk}.
$$
Here, $C_{ijk}=C(E_i, E_j, E_k)$ and 
we have used the fact that  $\sum_iC(E_i, Y, E_i)=0$ for any $Y$. By substituting these results, we obtain the desired equation:
$$
(n-2) \delta B (X) = -\frac {n-4}{n-2} \, C(X, E_j, E_k)z_{jk}.
$$
Finally, it is obvious from the definition of $C$ that 
$$\langle i_XC, z\rangle= \frac 12\, d|z|^2(X)+\delta (z\circ z)(X).$$
\end{proof}
By Proposition~\ref{lem02}, it is clear that $\delta B=0$ when $n=4$, and $\langle i_XC, z\rangle =0$ for any vector field $X$ if and only if $\delta B=0$ when $n\geq 5$.
In particular, since
$$0=\langle i_XC, z\rangle= \frac 12\, d|z|^2(X)+\delta (z\circ z)(X),$$
we have
\be  \frac 12 \, \tr |z|^2=\delta \delta (z\circ z). \label{eqn07} 
\ee

The following result holds when the scalar curvature is constant.
\begin{proposition} {\rm ((10) of \cite{ych1})} \label{lem1127} Assume that $s_g$ is constant. Then, 
$$\delta d^Dr = D^*Dz+\frac n{n-2}\, z\circ z +\frac {s}{n-1} z -\frac 1{n-2}|z|^2g -\mathring{\mathcal W}z.
$$
\end{proposition}
\begin{proof}
The proof follows from the identity (see 4.71 in \cite{Be})
$$\delta d^Dr = D^* Dr +\frac 12 Dds +r\circ r -\mathring{R}r, $$
and the relation in \cite{ych1}
$$ \mathring{R}r= \mathring{\mathcal W}z+\frac 1{n-2} |z|^2 g +\frac {(n-2)s}{n(n-1)}\, z -\frac 2{n-2}z\circ z +\frac {s^2}{n^2}g, $$
which  comes from the decomposition of the Riemann tensor.
\end{proof}

\section{ Critical metrics}
In this section, we turn our attention to a non-trivial solution $(g,f)$ of (\ref{cpe0}). 
To do this, we will introduce a covariant $3$-tensor $T$ defined by 
$$ T= \frac 1{n-2}\, df  \wedge z +\frac 1{(n-1)(n-2)}\, i_{\nabla f }z \wedge g,$$
Also we define the interior product $\tilde{i}$ to the final factor by
\be \tilde{i}_{V}\omega (X,Y, Z)= \omega(X,Y, Z, V) \label{lateri}\ee
for a $(4,0)$-tensor $\omega$ and a vector field $V$.

Now, from the critical metric equation (\ref{cpe0}) we have
\be (1+f)z = Ddf +\frac {sf}{n(n-1)}g.\label{cpe1}\ee
By applying $d^D$ to both sides of this equation and using the Ricci identity
$$ d^D Ddf(X,Y,Z)=R(X,Y, Z, \nabla f)$$
for any vector fields $X$, $Y$, $Z$ on $M$, we obtain
$$ (df \wedge z +(1+f)d^Dz)(X,Y,Z)=\tilde{i}_{\nabla f}R(X,Y,Z)+\frac s{n(n-1)}\, df \wedge g (X,Y,Z).$$
Since $C=d^Dz$ when $s$ is constant, we obtain 
\be (1+f)  \, C  =\tilde{i}_{\nabla f}R - df\wedge z +\frac s{n(n-1)}\, df \wedge g =  \tilde{i}_{\nabla f}{\mathcal W}-(n-1)\, T. \label{eqn09}\ee
Here, we used the fact that 
$$
\tilde{i}_{\nabla f}R= \tilde{i}_{\nabla f}{\mathcal W} -
\frac 1{n-2}i_{\nabla f}r \wedge g -\frac 1{n-2}df \wedge r +\frac s{(n-1)(n-2)}\, df \wedge g,
$$
which follows from the curvature decomposition (c.f. \cite[1.116, p.48]{Be})
\bea R(X,Y,Z,W)&=& {\mathcal W}(X,Y,Z,W)+\frac 1{n-2}(g(X,Z)r(Y,W)+g(Y,W)r(X,Z)\\
& &-g(Y,Z)r(X,W) -g(X,W)r(Y,Z))\\
& & -\frac s{(n-1)(n-2)}(g(X,Z)g(Y,W)-g(Y,Z)g(X,W)).
\eea

From the definition of the Bach tensor (\ref{eqn05}) and (\ref{eqn09}), we have 
\be
(n-2)B= -\delta C+\mathring{\mathcal W}z
= -\delta \left(  \frac 1{1+f }\, \tilde{i}_{\nabla f }{\mathcal W}  -(n-1)\frac T{1+f }\right)+\mathring{\mathcal W}z. 
\label{eqn10}\ee
Since
$$ \delta \tilde{i}_{\nabla f }{\mathcal W}(X,Y)= -\frac {n-3}{n-2} \, d^Dr(Y,\nabla f , X)+(1+f)  \,\mathring{\mathcal W}z(X,Y),$$
we have
\bea
 \lefteqn{-\delta \left( \frac 1{1+f}\, \tilde{i}_{\nabla f }{\mathcal W}\right)(X,Y)=
   -\frac 1{(1+f) ^2} \, {\mathcal W}(\nabla f , X, Y, \nabla f )   -\frac 1{1+f }\,  \delta \tilde{i}_{\nabla f }{\mathcal W}(X,Y)} \\ 
 &=& \frac 1{(1+f) ^2} \, {\mathcal W}(X, \nabla f , Y, \nabla f) +\frac {n-3}{n-2}\,\frac 1{1+f } d^Dr(Y, \nabla f , X) -\mathring{\mathcal W}z(X,Y).
\eea
Therefore, it follows from (\ref{eqn09}) and (\ref{eqn10}) that
\be
(n-2)(1+f)B(X,Y)
= C(X, \nabla f, Y)+\frac {n-3}{n-2}\, { C(Y, \nabla f , X)} + (n-1)\delta T(X,Y). \label{eqn11}\ee

On the other hand, by taking the divergence of $T$, we have
\begin{eqnarray}
(n-1)(n-2) \delta T(X,Y)&=& \frac {n-2}{n-1}sfz(X,Y) - (n-2) D_{\nabla f}z(X,Y)+C(Y, \nabla f, X)\nonumber \\
& & +n(1+f)z\circ z(X,Y)-(1+f) |z|^2g(X,Y).\label{eqn15}
\end{eqnarray}

By combining (\ref{eqn11}) and (\ref{eqn15}), we obtain the followings.
\begin{proposition} \label{glem3} On $M$, we have
\bea(n-1)\langle \delta T, z\rangle&=&  (n-2)(1+f)\langle B, z\rangle \nonumber  \\
&=&   \frac {sf}{n-1} |z|^2  -\frac 12 \nabla f(|z|^2)+\frac n{n-2}(1+f)\langle z\circ z, z\rangle.
\eea
\end{proposition}
Also we have
\begin{proposition} \label{glem5} On $M$, we have
$$(n-1)\delta \delta T (X)= \frac {n-1}{n-2}(1+f)\langle i_XC, z\rangle +(n-1) \langle i_XT, z\rangle.$$
\end{proposition}
\begin{proof}
By taking the derivative of (\ref{eqn11}), we have
\bea
(n-2)B(\nabla f, X)&=& (n-2)(1+f) \delta B(X)-\delta C(\nabla f, X)\\
& & +\frac {n-3}{n-2}(1+f) \langle i_XC, z\rangle -(n-1)\delta \delta T (X).
\eea 
Thus, by (\ref{eqn05}) and Proposition~\ref{lem02} we have
\bea (n-1)\delta \delta T (X)&=& -\mathring{\mathcal W}z(\nabla f, X) + \frac 1{n-2}(1+f)\langle i_X C,z\rangle\\
&=& \frac {n-1}{n-2}(1+f)\langle i_XC, z\rangle +(n-1) \langle i_XT, z\rangle.
\eea
where the last equality comes from (\ref{eqn09}).
\end{proof}
We also have the following.
\begin{lemma} \label{glem4} We have
$$|T|^2=\frac 2{n-2} \langle i_{\nabla f}T, z\rangle,$$
and
$$ \frac {(n-2)^2}2 |T|^2= |z|^2|\nabla f|^2-\frac n{n-1}z\circ z(\nabla f, \nabla f).$$
\end{lemma}
\begin{proof}
It is a straightforward computation. From the definition of $T$,
$$|T|^2=\frac 1{n-2}\sum_{i,j,k}T(E_i, E_j, E_k)\left(df\wedge z+\frac 1{n-1}i_{\nabla f}z\wedge g(E_i, E_j, E_k)\right)=\frac 2{n-2} \langle i_{\nabla f}T, z\rangle.$$
Also
\bea 
(n-2)^2|T|^2&=&\sum_{i,j,k} |df(E_i)z_{jk}-df(E_j)z_{ik}+\frac 1{n-1}\,\left( z(\nabla f, E_i)g_{jk}-z(\nabla f, E_j)g_{ik}\right)|^2\\
&=&2 |\nabla f|^2|z|^2 -\frac {2n}{n-1}\, z\circ z(\nabla f, \nabla f).
\eea
\end{proof}

\section{Proof of  Theorem~\ref{main1}}
In this section, we prove Theorem~\ref{main1}. Throughout the section and the next section, we assume that $\langle i_X C, z\rangle =0$ for any vector $X$ with  $n \geq 4$.  
To prove Theorem~\ref{main1}, we first need the following.
\begin{lemma} \label{glem1} Let $(g,f)$ be a non-trivial solution of (\ref{cpe0})  on an $n$-dimensional compact manifold $M$, $n\geq 4$. Assume that $\langle i_XC,z\rangle =0$. Then 
$$ \int_M (1+f)\langle z\circ z , z\rangle =\frac {(n-2)s}{2n(n-1)} \int_M |z|^2.
$$
\end{lemma}
\begin{proof} Note that
$$\frac 12 \int_M (1+f)\tr |z|^2=\frac 12 \int_M |z|^2\tr f = -\frac s{2(n-1)}\int_M f|z|^2.
$$
Also, by  (\ref{eqn07}) we have
\bea
\frac{1}{2}\int_M (1+f)\tr |z|^2 &=&
\int_M (1+f) \delta \delta(z\circ z) =\int_M \delta (z\circ z)(\nabla f)=  \int_M \langle z\circ z, Ddf\rangle\\
&=&
\int_M (1+f)\langle z\circ z, z\rangle - \frac{s}{n(n-1)}\int_M  f|z|^2.
\eea
Thus,
$$
\int_M (1+f)\langle z\circ z, z\rangle  = - \frac{(n-2)s}{2n(n-1)}\int_M f |z|^2.
$$
However, by (\ref{cpe0}) it is easy to see that 
$$\delta (i_{\nabla f}z)=-(1+f)|z|^2,$$
which implies that
\be \int_M f|z|^2=-\int_M |z|^2.\label{geq2}\ee
\end{proof}

The following is the Okumura inequality which can be found in Lemma 2.6 of \cite{Ok} (c.f. see also Lemma 2.4 of \cite{Hui}).
\begin{lemma} \label{lem2017-12-24-2}
For any real numbers $a_1, \cdots, a_n$ with $\displaystyle{\sum_{i=1}^n a_i = 0}$, 
we have
$$
- \frac{n-2}{\sqrt{n(n-1)}} \left(\sum_{i=1}^n a_i^2\right)^{3/2} \le \sum_{i=1}^n a_i^3
\le \frac{n-2}{\sqrt{n(n-1)}} \left(\sum_{i=1}^n a_i^2\right)^{3/2},
$$
and equality holds if and only if at least $n-1$ of the $a_i$'s are all equal.
\end{lemma}
\vskip .5pc

Now, we are ready to prove Theorem~\ref{main1}. 

Let $M_0:= \{ f\le -1\}$ and $M^0:= \{f >-1\}$. Note that $\langle z\circ z, z\rangle=\mbox{tr}(z^3)$. By applying Lemma~\ref{lem2017-12-24-2} to the traceless Ricci tensor $z$, we have
$$
(1+f)\langle z\circ z, z\rangle \le - \frac{n-2}{\sqrt{n(n-1)}} (1+f)|z|^3
$$
on the set $M_0$, and
$$
(1+f)\langle z\circ z, z\rangle \le  \frac{n-2}{\sqrt{n(n-1)}} (1+f)|z|^3
$$
on the set $M^0$. 

Therefore, by Lemma~\ref{glem1} with $M=M_0\cup  M^0$,
\bea
\frac{(n-2)s}{2n(n-1)}\int_M  |z|^2 
&=&  \int_{M_0} (1+f)\langle z\circ z, z\rangle  + \int_{M^0} (1+f)\langle z\circ z, z\rangle\\
&\le&
- \frac{n-2}{\sqrt{n(n-1)}} \int_{M_0} (1+f)|z|^3
+  \frac{n-2}{\sqrt{n(n-1)}} \int_{M^0} (1+f)|z|^3\\
&\le&
 \frac{n-2}{\sqrt{n(n-1)}} |\min_M(1+f)| \int_{M_0}|z|^3
 +  \frac{n-2}{\sqrt{n(n-1)}} \, \max_M (1+f) \int_{M^0} |z|^3\\
  &\le&
   \frac{(n-2)\mu }{\sqrt{n(n-1)}} \int_M|z|^3.
\eea
Recall that $
\mu= \max\{|\min_M(1+f)|, \max_M (1+f)\,\}
$.  

Consequently, we obtain
$$
\frac{n-2}{\sqrt{n(n-1)}} \int_M \left(\frac s{2\sqrt{n(n-1)}} - \mu\, |z|\right) |z|^2\le 0.\label{eqn2017-12-26-5}
$$
From the assumption, $$\frac s{2\sqrt{n(n-1)}}- {\mu}|z| \ge 0 .$$ As a result, we have either
$
z= 0$, or $$|z| = \frac{s}{2\mu\sqrt{n(n-1)}}.
$$
From (\ref{geq2}) and the fact that $\int_M f= 0$, the second case should be excluded; otherwise 
$$0=\int_M (1+f)|z|^2= \frac {s^2}{4n(n-1)\mu^2}\int_M (1+f)=\frac {s^2}{4n(n-1)\mu^2},
$$
which is a contradiction.
\hfill $\Box$

\section{Proof of  Theorem~\ref{main2}}
In this section, we prove Theorem~\ref{main2}.  To do this, we first show the following integral identity.

\begin{lemma}
We have
\be
\frac {(n-1)(n-2)}2\int_M |T|^2
&=& \frac {s}{n}\int_M f^2|z|^2 + \int_M  z\circ z(\nabla f, \nabla f) \label{eqn2018-1-22-2}\\
&+& \frac {2(n-1)}{n-2}\int_M f(1+f)\langle z\circ z, z\rangle. \nonumber
\ee
\end{lemma}
\begin{proof}
It follows from Proposition~\ref{glem5} and Lemma~\ref{glem4} together with the assumption 
 $\langle i_XC, z\rangle =0$ for any vector $X$ that
$$ 
\delta \delta T(\nabla f)=\langle i_{\nabla f}T, z\rangle = \frac {n-2}2\, |T|^2.
$$
Thus,
\be
\int_M f\delta \delta \delta T =\int_M \delta \delta T(\nabla f)=\frac {n-2}2 \int_M |T|^2.\label{seq1-1}
\ee
By Proposition~\ref{glem5} again with $\langle i_XC, z\rangle =0$, we have
\be \delta \delta T (X)=\langle i_XT, z\rangle,\label{1122e}\ee
From the definition of $T$,
$$
\langle T, C\rangle =\frac 1{n-2}\langle df\wedge z, C\rangle = \frac 2{n-2}\langle i_{\nabla f}C, z\rangle =0,
$$
and so, by taking the divergence of $\delta \delta T$,  it follows from (\ref{1122e}) that
$$
\delta \delta \delta T =\langle \delta T, z\rangle -\frac 12 \langle T, C\rangle= \langle \delta T, z\rangle.
$$
Thus, by Proposition~\ref{glem3}
\bea
(n-1)\delta \delta \delta T =\frac {sf}{n-1}|z|^2-\frac 12 \nabla f(|z|^2)
+ \frac n{n-2}(1+f)\langle z\circ z , z\rangle.\label{seq2-1}
\eea
From this together with (\ref{seq1-1}), we have
\be
\frac {(n-1)(n-2)}2\int_M |T|^2
&=& \frac {s}{n-1}\int_M f^2|z|^2-\frac 12 \int_M f\nabla f(|z|^2) \label{eqn2018-1-23-1}\\
&&+\frac n{n-2}\int_M f(1+f)\langle z\circ z, z\rangle. \nonumber
\ee
Next,  by Proposition~\ref{lem02} with the assumption that $\langle i_XC,z\rangle =0$,  we have
\bea
\frac 12 \int_M f\nabla f(|z|^2)&=& -\int_M \delta (z\circ z)(f\nabla f)\\
&=& -\int_M  z\circ z (\nabla f, \nabla f) -\int_M f\langle z\circ z, Ddf\rangle\\
&=&-\int_M z\circ z (\nabla f, \nabla f) -\int_M f(1+f)\langle z\circ z, z\rangle +\frac s{n(n-1)}\int_M f^2|z|^2.
\eea
Here, the last equality comes from (\ref{cpe1}). Substituting this into (\ref{eqn2018-1-23-1}), 
we obtain (\ref{eqn2018-1-22-2}).
\end{proof}

Now, we are ready to prove Theorem~\ref{main2}.
\begin{proof}
By Lemma~\ref{glem4} 
\bea 
\frac {(n-1)(n-2)}2\int_M |T|^2&=& \frac {n-1}{n-2}\int_M |z|^2|\nabla f|^2 
 - \frac n{n-2}\int_M z\circ z(\nabla f, \nabla f).
\eea
Comparing this to (\ref{eqn2018-1-22-2}), we have
\bea
\lefteqn{ \frac {n-1}{n-2}\int_M |z|^2|\nabla f|^2 -\frac {2(n-1)}{n-2} \int_M z\circ z (\nabla f, \nabla f)}\\
&&\quad
=\frac sn \int_M  f^2|z|^2 +\frac {2(n-1)}{n-2}\int_M f(1+f)\langle z\circ z, z\rangle.
\eea
From the assumption 
$$ |z|^2\leq \min \left\{ 2|i_Nz|^2 , \frac {s^2}{4n(n-1)} \right\},  $$
we have 
$$|\nabla f|^2|z|^2 \leq 2|i_{\nabla f}z|^2=2z\circ z(\nabla f, \nabla f),$$
and so
$$\frac sn \int_M  f^2|z|^2 +\frac {2(n-1)}{n-2}\int_M f(1+f)\langle z\circ z, z\rangle\leq 0.$$

Note that, by Lemma~\ref{glem1}
\bea 
\int_M f(1+f)\langle z\circ z, z\rangle &=&\int_M f^2 \langle z\circ z, z\rangle +\int_M f\langle z\circ z, z\rangle\\
&=&\int_M f^2 \langle z\circ z, z\rangle +\frac {(n-2)s}{2n(n-1)}\int_M |z|^2 -\int_M \langle z\circ z,z\rangle.
\eea
Therefore, applying Lemma~\ref{lem2017-12-24-2}
$$
\frac sn\int_M (1+f^2)|z|^2\leq \frac {2(n-1)}{n-2}\int_M (1-f^2)\langle z\circ z, z\rangle\leq  \frac {2\sqrt{n-1}}{\sqrt{n}}\int_M (1+f^2) |z|^3,
$$
which implies 
$$ 
0\leq  \int_M (1+f^2)|z|^2\left( \frac s{2\sqrt{n(n-1)}} -|z|\right) \leq 0,
$$
where the first inequality follows from the assumption  
$$ 
|z|\leq \frac s{2\sqrt{n(n-1)}}.
$$
Hence, we may conclude that $z=0$ on all of $M$. If the equality holds,
$$ |z|= \frac s{2\sqrt{n(n-1)}}$$
then we reach a contradiction as in the proof of Theorem~\ref{main1}.
\end{proof}

\end{document}